\documentclass[reqno,11pt]{amsart}
\usepackage{amsmath,amsfonts,amssymb,amsthm,graphicx,color}
\usepackage{mathrsfs}
\usepackage{yhmath}
\usepackage{xcolor}

 \allowdisplaybreaks
\numberwithin{equation}{section}

\def\be{\begin{equation}}
\def\ee{\end{equation}}
\def\bea{\begin{eqnarray}}
\def\eea{\end{eqnarray}}

\def\eps{\varepsilon}
\def\e{\varepsilon}

\def\R{{\mathbb R}}
\def\T{{\mathbb T}}

\usepackage{mathtools}
\mathtoolsset{showonlyrefs}

\newcommand{\br}{\mathbb{R}}
\newcommand{\bz}{\mathbb{Z}}

\newcommand{\bt}{\mathbb{T}}

\newcommand{\W}{KW}

\renewcommand{\(}{\left(}
\renewcommand{\)}{\right)}
\renewcommand{\[}{\left[}
\renewcommand{\]}{\right]}
\newcommand{\na}{\nabla}

\def\bal{\begin{aligned}}
\def\eal{\end{aligned}}

\def\proofof#1{\begin{proof}[Proof of #1]}

\def\part#1#2{\par\noindent{\underline{\it Part~#1.}}\emph{ #2}\\}

\def\pt{\partial}
\def\({\left(}
\def\){\right)}
\def\[{\left[}
\def\]{\right]}
\def\na{\nabla}

\def\W{W}
\newcommand*\di{\mathop{}\!\mathrm{d}}

\def\XXint#1#2#3{{\setbox0=\hbox{$#1{#2#3}{\int}$} \vcenter{\vspace{-1pt}\hbox{$#2#3$}}\kern-.5\wd0}}

\theoremstyle{plain}
\newtheorem{lemma}{Lemma}[section]

\newtheorem{theorem}[lemma]{Theorem}

\newtheorem{corol}[lemma]{Corollary}
\newtheorem{defin}[lemma]{Definition}
\newtheorem{remark}[lemma]{Remark}

\def\de0#1{\rule[3pt]{#1}{0.4pt} \hspace{-0.1pt} \rule[3.05pt]{0.05pt}{0.4pt} \hspace{-0.1pt} \rule[3.1pt]{0.05pt}{0.4pt} \hspace{-0.1pt} \rule[3.15pt]{0.05pt}{0.4pt} \hspace{-0.1pt} \rule[3.2pt]{0.05pt}{0.4pt} \hspace{-0.1pt} \rule[3.25pt]{0.05pt}{0.4pt} \hspace{-0.1pt} \rule[3.3pt]{0.05pt}{0.4pt} \hspace{-0.1pt} \rule[3.35pt]{0.05pt}{0.4pt} \hspace{-0.1pt} \rule[3.4pt]{0.05pt}{0.4pt} \hspace{-0.1pt} \rule[3.45pt]{0.05pt}{0.4pt} \hspace{-0.1pt} \rule[3.5pt]{0.05pt}{0.4pt} \hspace{-0.1pt} \rule[3.55pt]{0.05pt}{0.4pt} \hspace{-0.1pt} \rule[3.6pt]{0.05pt}{0.4pt} \hspace{-0.1pt} \rule[3.65pt]{0.05pt}{0.4pt} \hspace{-0.1pt} \rule[3.7pt]{0.05pt}{0.4pt} \hspace{-0.1pt} \rule[3.75pt]{0.05pt}{0.4pt} \hspace{-0.1pt} \rule[3.8pt]{0.05pt}{0.4pt} \hspace{-0.1pt} \rule[3.85pt]{0.05pt}{0.4pt} \hspace{-0.1pt} \rule[3.9pt]{0.05pt}{0.4pt} \hspace{-0.1pt} \rule[3.95pt]{0.05pt}{0.4pt} \hspace{-0.1pt} \rule[4.0pt]{0.05pt}{0.4pt} \hspace{-0.1pt} \rule[4.05pt]{0.05pt}{0.4pt} \hspace{-0.1pt} \rule[4.1pt]{0.05pt}{0.4pt} \hspace{-0.1pt} \rule[4.15pt]{0.05pt}{0.4pt} \hspace{-0.1pt} \rule[4.2pt]{0.05pt}{0.4pt} \hspace{-0.1pt} \rule[4.25pt]{0.05pt}{0.4pt} \hspace{-0.1pt} \rule[4.3pt]{0.05pt}{0.4pt} \hspace{-0.1pt} \rule[4.35pt]{0.05pt}{0.4pt} \hspace{-0.1pt} \rule[4.4pt]{0.05pt}{0.4pt} \hspace{-0.1pt} \rule[4.45pt]{0.05pt}{0.4pt} \hspace{-0.1pt} \rule[4.5pt]{0.05pt}{0.4pt} \hspace{-0.1pt} \rule[4.55pt]{0.05pt}{0.4pt} \hspace{-0.1pt} \rule[4.6pt]{0.05pt}{0.4pt} \hspace{-0.1pt} \rule[4.65pt]{0.05pt}{0.4pt} \hspace{-0.1pt} \rule[4.7pt]{0.05pt}{0.4pt} \hspace{-0.1pt} \rule[4.75pt]{0.05pt}{0.4pt} \hspace{-0.1pt} \rule[4.8pt]{0.05pt}{0.4pt} \hspace{-0.1pt} \rule[4.85pt]{0.05pt}{0.4pt} \hspace{-0.1pt} \rule[4.9pt]{0.05pt}{0.4pt} \hspace{-0.1pt} \rule[4.95pt]{0.05pt}{0.4pt} \hspace{-0.1pt} \rule[5.0pt]{0.05pt}{0.4pt} \hspace{-0.1pt} \rule[5.05pt]{0.05pt}{0.4pt} \hspace{-0.1pt} \rule[5.1pt]{0.05pt}{0.4pt} \hspace{-0.1pt} \rule[5.15pt]{0.05pt}{0.4pt} \hspace{-0.1pt} \rule[5.2pt]{0.05pt}{0.4pt} \hspace{-0.1pt} \rule[5.25pt]{0.05pt}{0.4pt} \hspace{-0.1pt} \rule[5.3pt]{0.05pt}{0.4pt} \hspace{-0.1pt} \rule[5.35pt]{0.05pt}{0.4pt} \hspace{-0.1pt} \rule[5.4pt]{0.05pt}{0.4pt} \hspace{-0.1pt} \rule[5.45pt]{0.05pt}{0.4pt} \hspace{-0.1pt} \rule[5.5pt]{0.05pt}{0.4pt} \hspace{-0.1pt} \rule[5.55pt]{0.05pt}{0.4pt} \hspace{-0.1pt} \rule[5.6pt]{0.05pt}{0.4pt} \hspace{-0.1pt} \rule[5.65pt]{0.05pt}{0.4pt} \hspace{-0.1pt} \rule[5.7pt]{0.05pt}{0.4pt} \hspace{-0.1pt} \rule[5.75pt]{0.05pt}{0.4pt} \hspace{-0.1pt} \rule[5.8pt]{0.05pt}{0.4pt} \hspace{-0.1pt} \rule[5.85pt]{0.05pt}{0.4pt} \hspace{-0.1pt} \rule[5.9pt]{0.05pt}{0.4pt} \hspace{-0.1pt} \rule[5.95pt]{0.05pt}{0.4pt} \hspace{-0.1pt} \rule[6.0pt]{0.05pt}{0.4pt}}	%fundamental brick
\newcounter{mt}

\begin{document}

\title[A new perspective on Wasserstein distances for kinetic problems]{A new perspective on Wasserstein distances\\ for kinetic problems}

\author{Mikaela Iacobelli}
\thanks{ETH Z\"urich, Department of Mathematics, R\"amistrasse 101, 8092 Z\"urich, Switzerland.\newline Email: \textsf{mikaela.iacobelli@math.ethz.ch}}

\begin{abstract}
We introduce a new class of Wasserstein-type distances specifically designed to tackle questions concerning stability and convergence to equilibria for kinetic equations.
Thanks to these new distances, we improve some classical estimates by Loeper \cite{Loep} and Dobrushin \cite{Dob} on Vlasov-type equations, and we present an application to quasi-neutral limits.
\end{abstract}

\maketitle

\section{Introduction}

\subsection{General overview}
Monge-Kantorovich distances, also known as Wasserstein distances, play a central role in statistical mechanics, especially in the theory of propagation of chaos and studying large particle systems' mean behavior.
%These distances are a crucial tool in many arguments in kinetic theory. In particular, they are frequently used to prove the uniqueness and stability of solutions to kinetic equations, study singular limits, and measure convergence to equilibrium.
From the late 1970s, there have been many applications of Wasserstein distances in kinetic theory, as is beautifully described in the bibliographical notes of \cite[Chapter 6]{Vil09}.
In particular, these distances are frequently used to prove the uniqueness and stability of solutions to kinetic equations, study singular limits, and measure convergence to equilibrium.

The first celebrated result relying on Monge-Kantorovich-Wasserstein distances in non-collisional kinetic theory is the proof by Dobrushin \cite{Dob} on the well-posedness for Vlasov equations with $C^{1,1}$ potentials, where existence, uniqueness, and stability are proved via a fixed point argument in the bounded-Lipschitz or the $1$-Wasserstein distance.
As a consequence of this argument, one also obtains the validity of the mean-field limit for Vlasov equations with smooth potentials.
The interested reader may refer to \cite[Chapter 1.4]{GolseReview2016} and \cite[Chapter 3.3]{JabinReviewMFL} for a detailed explanation of Dobrushin's stability estimate, its consequences on the mean-field limit for the Vlasov equation, and of the role of Monge-Kantorovich-Wasserstein distances.
Dobrushin's estimate is at the core of several kinetic theory arguments; see for example \cite{CanizoCarrilloRosado2011,CagliotiRousset2008, CarrilloChoiHauray2014, CarrilloChoiSalem2019, DeBievreGoudonVavasseur2016, Xavipatit2018, GolseMouhotRicci2013, GPI-WP} for some applications.

In recent times, Golse and Paul in \cite{GolsePaul2017} introduced a quantum analog of the $2$-Wasserstein distance to measure the approximation of the $N$-body quantum dynamics by its mean-field limit. 
In \cite{GolseMouhotPaul2016} the authors prove quantitative stability estimates that are reminiscent of Dobrushin's, and they show that, in the case of $C^{1,1}$ potentials, the mean-field limit of the quantum mechanics of $N$ identical particles is uniform in the classical limit.

Another fundamental stability estimate was proved by Loeper \cite{Loep}, who established uniqueness and stability of solutions with bounded density for the Vlasov-Poisson equation. Loeper's argument relies on the fact that the Coulomb kernel is generated by a potential solving Poisson's equation and exploits the strong connection between the $2$-Wasserstein distance and the $H^{-1}$-norm. Besides providing the best-known uniqueness criterion for Vlasov-Poisson, this approach also gives a new proof of uniqueness \`a la Yudovich for $2D$ Euler. Loeper's result has been generalized to less singular kernels \cite{HaurayWass2009}, 
and it is the cornerstone for several other stability arguments \cite{BertozziLaurentRosado2011,CarrilloRosado2010, Han-KwanMiotMoussaMoyano, Lafleche2019, LewinSabin2020, LoeperSIMA2006, SerfatyVazquez2014}. 
Also, Loeper's uniqueness criterion for Vlasov-Poisson has been extended to solutions whose associated density belongs to some suitable Orlicz spaces \cite{MiotCMP2016, HoldingMiot2018}.
In the following, we will focus our attention on some applications of Loeper's stability estimate related to quasi-neutral limit for the Vlasov-Poisson equation \cite{GPI-SIAM2018, GPI20, HKI2, HKI1}.

In general, extending Dobrushin's and Loeper's estimates is a delicate matter. A possible idea is to introduce an anisotropic metric that weights spatial and momentum coordinates differently. For example, in \cite{Laz}, the author considers a variant of the $2$-Wasserstein distance where the cost for moving points in the $x$-variable is higher than for the $v$-variable.
By suitably selecting the parameters, this allows the author to extend the validity ranges for the mean-field limit for the Vlasov-Poisson system.
Also, as shown in \cite{GPI-SIAM2018, GPI20}, an analogous method provides better convergence estimates when considering combined mean-field and quasi-neutral limits in Vlasov-Poisson-type systems.
At the same time as this paper was written, another variant of this idea was introduced in \cite{Salem2021}, where the author improves the trend to equilibrium for $1$-D kinetic Fokker-Planck equations via estimates measured in an analog of the $2$-Wasserstein metric.

This work aims to push further the idea that, when applied to kinetic problems, Wasserstein distances should be modified to reflect the natural anisotropy between position and momentum variables. Moreover, since these metrics are used to measure the distance between PDEs' solutions, we will introduce time-dependent counterparts that can vary along with the characteristic flow. 
Still, it is worth noticing that our method could be applied, beyond the kinetic framework, to equations where the evolution in one of the variables enjoys better regularity properties than the others.

\smallskip

Before stating our main results, let us emphasize that the idea of finding appropriate generalised Wasserstein distances has been used successfully in other contexts in the optimal transport and evolution PDE community, see for instance \cite{DolbeaultNazaretSavare2009, FigalliGigli2010, LieroMielkeSavare2016, LieroMielkeSavare2018,PiccoliRossi2014, PiccoliRossi2016} and references therein.

\subsection{Definitions and main results}

Let us recall the definition of Wasserstein distances (see for instance \cite{AmbrosioGigliSavareBOOK, Vil09}).  
In what follows, $\mathcal X$ will be either the $d$-dimensional torus $\mathbb T^d$ or the Euclidean space $\br^d$.

\begin{defin}\label{def:Wp} Let $\mu, \nu$ be two probability measures on $\mathcal X\times \br^d$.
We denote with $\Pi(\mu, \nu)$ the set of all probability measures on $(\mathcal X\times \br^d)^2$ with \emph{marginals} $\mu$ and $\nu.$   More precisely, $\pi \in \Pi(\mu, \nu)$ if 
$$
\pi[A \times (\mathcal X\times \br^d)]=\mu[A], \ \ \ \pi[(\mathcal X\times \br^d) \times B]= \nu[B], \qquad \text{for all $A,B\subset \mathcal X\times \br^d$ Borel.} 
$$
We shall call \emph{coupling} (between $\mu$ and $\nu$) an element in $\Pi(\mu, \nu).$

For $p\geq 1$, the $p$-Wasserstein distance between $\mu$ and $\nu$ is defined as
$$
W_{p} (\mu, \nu):=\left( \inf_{\gamma \in \Pi(\mu,\nu)} \int_{(\mathcal{X} \times \mathbb{R}^d)^2} \(|x-y|^{p} +|v-w|^p\)\mathrm{d} \gamma (x,v,y,w) \right)^{1/p}.
$$
\end{defin}

\noindent {\bf (1) A free-flow $W_1$-type distance for the Vlasov equations with $C^{1,1}$ potential.}
Consider two solutions $f_1,f_2$ of the Vlasov equation on $\mathcal X$, namely
$$
\partial_t f+v\cdot \nabla_x f+F[f]\cdot \nabla_vf=0,\qquad F[f]:=\nabla K\ast \rho_f,\qquad \rho_f:=\int f dv,
$$
where $\|D^2 K\|_{\infty}=:B<\infty$.
The classical Dobrushin's argument shows that
$$
W_1(f_1(t),f_2(t))\leq e^{(1+2B)t}W_1(f_1(0),f_2(0)).
$$
In particular, when the potential $K$ is identically zero, this bound provides an exponential stability for $W_1$ that is far from optimal. Indeed, since the solution is simply given by $f(t,x,v)=f(0,x-tv,v)$,
it is clear that in this case $W_1(f_1(t),f_2(t))\sim t$ for $t \gg 1$.

By introducing a $W_1$-type distance adapted to the free flow, we can prove that
$$
W_1(f_1(t),f_2(t))\leq \min\left\{(1+t) e^{\frac23 B\left((1+t)^3-1\right)},e^{(1+2B)t}\right\}W_1(f_1(0),f_2(0))
$$
(see Theorem~\ref{thm:Dobrushin Mika} below).

This estimate gives the optimal bound when $K\equiv 0$. Moreover, for $B\leq 1$, this provides a better estimate compared to the usual Dobrushin's bound when $t \in [0,T_B]$ with $T_B\simeq B^{-1/2}$. 

\noindent {\bf (2) An improved $W_2$-stability estimate for Vlasov-Poisson with bounded density.}
For this second application, we focus on the case of the torus for simplicity, but a completely similar analysis works on the whole space.

Consider two solutions $f_1,f_2$ of the Vlasov-Poisson equation on $\mathbb T^d$, namely
$$
\partial_t f+v\cdot \nabla_x f+\nabla U\cdot \nabla_vf=0,\qquad -\Delta U:= \rho_f-1,\qquad \rho_f:=\int f dv.
$$
As shown in \cite{Loep}, Loeper's proof provides the following stability estimate whenever $W_2(f_1(0),f_2(0))$ is sufficiently small (which is the interesting case):
$$
W_2(f_1(t),f_2(t))\leq c_d e^{\log\left(\frac{W_2(f_1(0),f_2(0))}{c_d}\right) e^{-Ct}},
$$
where $c_d>0$ is a dimensional constant, while $C$ depends on the $L^\infty$ norm of $\rho_{f_1}$ and $\rho_{f_2}.$
This estimate can then be applied to prove the validity of the quasi-neutral limit for Vlasov-Poisson
for initial data that are double-exponential perturbation of analytic functions \cite{HKI2, HKI1} (see Remark~\ref{rem:double}).

To improve this result, given $(X_i,V_i)$ the characteristics associated to $f_i$,
we consider a {\em nonlinear $W_2$-type quantity} of the form
$$
Q(t):=\int_{(\bt^d\times\br^d)^2} \[\lambda(t)|X_1(t,x,v)-X_2(t,y,w)|^2+|V_1(t,x,v)-V_2(t,y,w)|^2\]d\pi_0(x,v,y,w)
$$
where $\pi_0$ is an optimal coupling, and $\lambda(t)=|\log(Q(t))|.$
We then prove that $Q(t)$ is well-defined whenever $Q(t)\ll 1,$  and finally, comparing $Q(t)$ to $W_2,$ we show that
$$
W_2(f_1(t),f_2(t))^2 \leq 2 e^{-\left(\sqrt{\left|\log\left\{W_2(f_1(0),f_2(0))^2 \left| \log \(\frac{1}2 W_2(f_1(0),f_2(0))^2\)\right|\right\}\right|} -Ct \right)^2}. 
$$
(see Theorem~\ref{thm:Loeper}).
To better understand the improvement of our estimate with respect to Loeper's, one can think as follows:
if $W_2(f_1(0),f_2(0)= \theta \ll 1$, then Loeper's estimate implies that $W_2(f_1(t),f_2(t)) \lesssim 1$ for $t \in [0,\log|\log\theta|]$.
Instead, our bound gives $W_2(f_1(t),f_2(t)) \lesssim 1$ for $t \in \bigl[0,|\log\theta|^{1/2}\bigr]$, so on a much longer time-interval.

\begin{remark}
\begin{itemize}
\item
Note that a standard Gronwall estimate of the form $W_2(f_1(t),f_2(t))\leq e^{Ct}W_2(f_1(0),f_2(0))$ would imply that $W_2(f_1(t),f_2(t)) \lesssim 1$ for $t \in [0,|\log\theta|]$.
So, while Loeper's bound loses an extra logarithm in terms of time-scale, our bound only loses a square root. Since the electric field for a solution with bounded density is at most log-Lipschitz, an estimate of the form $W_2(f_1(t),f_2(t))\leq e^{Ct}W_2(f_1(0),f_2(0))$ 
is not expected to hold in this setting, and we believe our bound to be essentially sharp.
\item
Our improvement from $\log|\log\theta|$ to $|\log\theta|^{1/2}$ is similar to the one obtained for the $W_1$ distance, see \cite[Remark 1.7]{HoldingMiot2018}.
In that paper, the authors rely crucially on the second-order structure of the Vlasov equation, namely $\ddot X=\nabla U(t,X)$.
Our proof, instead, relies only on the fact that $\dot X=a(t,X,V)$, where $a(t,\cdot,\cdot)$ is Lipschitz, and it can be generalized to other contexts where the second-order structure fails.
\end{itemize}
\end{remark}

Our new stability estimate has interesting applications for what concerns some singular limits for Vlasov-type equations.
In particular, by considering the Vlasov-Poisson system in appropriate dimensionless variables that take the Debye length into account, we prove the validity of the quasi-neutral limit for Vlasov-Poisson
for initial data that are an exponential perturbation of analytic functions, see also Remark~\ref{rem:double}.

\bigskip

The paper is structured as follows: in the next two sections, we will present our two main results, and then in the final section of the paper, we will discuss more generally our approach and how it leads to the introduction of a new family of Wasserstein-type distances.

\section{Dobrushin's estimate revisited}
\subsection{The Vlasov equation}
The Vlasov equation is a non-linear partial differential equation providing a statistical description for the collective behavior of large numbers of charged particles in mutual, long-range interaction. 
This model was first introduced by Jeans in the context of Newtonian stellar dynamics \cite{Jeans}, and later by Vlasov in his work on plasma physics \cite{Vla1, Vla2}. 
The unknown of the Vlasov equation $f(t,x,v)$ is the distribution function of the system at time, that is the number density of particles that are located at the position $x$ and have instantaneous velocity $v$ at time $t.$
The Vlasov equation for the distribution function $f$ reads as follows:

\begin{equation}\label{vlasov}
\partial_t f(t,x,v)+v\cdot \nabla_xf(t,x,v) +F[f](t,x) \cdot \nabla_v f(t,x,v)=0
\end{equation}
where
$$
F[f](x)= \iint \nabla K(x-y)\,f(dy,dw)=\na K\ast_{x,v} f.
$$
In other words, the Vlasov equation for particle systems is a kinetic model where each particle is subject to the acceleration field $F[f]$ created by all the other particles in the system.

The Vlasov equation is a transport equation and, for a sufficiently regular force field, it can be described by the method of characteristics. The initial distribution $f_0$ is transported by a characteristic flow $(X,V)$ generated by the mean-field force $F[f]$: if we denote 
$$
\left\{
\begin{array}{l}
\dot X(t,x,v)=V(t,x,v),\\
\dot V(t,x,v)=F[f](t,X(t,x,v)),\\
X(0,x,v)=x,\,V(0,x,v)=v,
\end{array}
\right.
$$
then $f(t,X(t,x,v),V(t,x,v))=f(0,x,v)$. Also, since 
 the vector field $(v, F[f])$ is divergence free, one has conservation of mass and of all $L^p$-norms.
For an introduction to this topic we refer to the lecture notes \cite{GolseReview2016}.

%The nonlinear Vlasov equation is a \emph{transport equation} and can be described by the \emph{method of characteristics}: if $f(t,x,v)$ solves the equation, then the measure $f(t,x,v)dx\,dv$ is the \emph{push-forward} of the initial measure $f_0(x,v)dx\,dv$ by the flow $\Phi_t=(X_t, V_t)$ in phase space, solving the characteristic equations:
%$$
%\left\{ \begin{array}{lll}
%\dot{X}_t=V_t\\ 
%\dot{V}_t=F(X_t,t)\ \ \mbox{with}\ \ F(X_t,t)= -\na W * \r \\
%(X_0,V_0)=(x,v).
%\end{array}
%\right.
%$$

\subsection{An improved Dobrushin's estimate}
Consider the Vlasov equation with smooth kernel. More precisely,
\be
\label{eq:Vlasov K}
\partial_t f+v\cdot \nabla_x f+F[f]\cdot \nabla_vf=0,\qquad F[f]:=\nabla K\ast \rho_f,\qquad \rho_f:=\int f dv,
\ee
where $\|D^2 K\|_{\infty}=:B<\infty$.
As explained in the introduction, our goal is to provide a stability estimate for solutions that is optimal in the regime as $B$ tends to zero.
Here is our result:
\begin{theorem}
\label{thm:Dobrushin Mika}
Let $f_1,f_2$ be two solution of \eqref{eq:Vlasov K}. Then
$$
W_1(f_1(t),f_2(t))\leq \min\left\{(1+t) e^{\frac23 B\left((1+t)^3-1\right)},e^{(1+2B)t}\right\}W_1(f_1(0),f_2(0)).
$$
\end{theorem}

\begin{proof}
Let $(X_i,V_i)$ denote the characteristic flow associated to $f_i$, that is
$$
\left\{
\begin{array}{l}
\dot X_i(t,x,v)=V_i(t,x,v),\\
\dot V_i(t,x,v)=\nabla \bigl(K\ast \rho_{f_i}\bigr)(t,X_i(t,x,v)),\\
X_i(0,x,v)=x,\,V_i(0,x,v)=v.
\end{array}
\right.
$$
Note that, since $\nabla K$ is Lipschitz, the characteristic flow is well-defined thanks to Cauchy-Lipschitz theory (see \cite[Chapter 2]{GolseReview2016}).
To prove Theorem~\ref{thm:Dobrushin Mika}, we consider $\pi_0$ an optimal $W_1$-coupling between $f_{1}(0)$ and $f_{2}(0)$,
and we define the quantity
$$
Q(t):= \int_{(\bt^d\times\br^d)^2} \big[|X_1(t,x,v)-t V_1(t,x,v)-(X_2(t,y,w)-tV_2(t,x,v))|+|V_1(t,x,v)-V_2(t,y,w)|\big]d\pi_0(x,v,y,w).
$$
Note that
\be
\label{eq:Q W1 0}
Q(0)=\int_{(\bt^d\times\br^d)^2} \[|x-y|+|v-w|\]d\pi_0(x,v,y,w)=W_1(f_1(0),f_2(0)).
\ee
Also
\be
\label{eq:Q W1}
\begin{split}
 &\int_{(\bt^d\times\br^d)^2} \big[|X_1(t,x,v)-X_2(t,y,w)|+|V_1(t,x,v)-V_2(t,y,w)|\big]d\pi_0(x,v,y,w)\\
 &\leq \int_{(\bt^d\times\br^d)^2} \big[|X_1(t,x,v)-t V_1(t,x,v)-(X_2(t,y,w)-tV_2(t,x,v))|+|V_1(t,x,v)-V_2(t,y,w)|\big]d\pi_0(x,v,y,w)\\
 &\qquad +t \int_{(\bt^d\times\br^d)^2} |V_1(t,x,v)-V_2(t,y,w)|d\pi_0(x,v,y,w)\\
&\leq (1+t)Q(t).
\end{split}
\ee
Since $\frac{d}{dt}(X_i-tV_i)=-t\dot V_i$, one has
\begin{align}
\frac{d}{dt}Q(t)&\leq (1+t)\int_{(\bt^d\times\br^d)^2} \[|\dot V_1(t,x,v)-\dot V_2(t,y,w)|\]d\pi_0(x,v,y,w)\\
&=(1+t)\int_{(\bt^d\times\br^d)^2} \left|\nabla \bigl(K\ast \rho_{f_1}\bigr)(t,X_1(t,x,v)) - \nabla \bigl(K\ast \rho_{f_2}\bigr)(t,X_2(t,y,w))\right|d\pi_0(x,v,y,w)\\
&\leq (1+t)\int_{(\bt^d\times\br^d)^2} \left|\nabla \bigl(K\ast \rho_{f_1}\bigr)(t,X_1(t,x,v)) - \nabla \bigl(K\ast \rho_{f_2}\bigr)(t,X_1(t,x,v))\right|d\pi_0(x,v,y,w)\\
&\qquad +(1+t)\int_{(\bt^d\times\br^d)^2} \left|\nabla \bigl(K\ast \rho_{f_2}\bigr)(t,X_1(t,x,v)) - \nabla \bigl(K\ast \rho_{f_2}\bigr)(t,X_2(t,y,w))\right|d\pi_0(x,v,y,w)\\
&=:(1+t)\bigl[T_1+T_2\bigr].
\end{align}
We now observe that, since $\nabla K$ is $B$-Lipschitz, we can bound
$$
T_2\leq B\int_{(\bt^d\times\br^d)^2} \big[|X_1(t,x,v)-X_2(t,y,w)|\big]d\pi_0(x,v,y,w) \leq B(1+t)Q(t),
$$
where the second inequality follows from \eqref{eq:Q W1}.
For $T_1$, we note that
\begin{align}
\left|\nabla \bigl(K\ast \rho_{f_1}\bigr)(t,X_1(t,x,v)) - \nabla \bigl(K\ast \rho_{f_2}\bigr)(t,X_1(t,x,v))\right|
=\left| \int_{\mathbb{T}^d} \nabla K(X_1(t,x,v)-z)d\,\bigl(\rho_{f_1(t)}(z)-\rho_{f_2(t)}(z)\bigr) \right|.
\end{align}
Here, similarly to Dobrushin's argument, we use that $W_1$ admits the following dual formulation:
\be
\label{eq:dual W1}
W_1(\mu,\nu)=\sup_{\psi \,\text{1-Lip}}\int \psi\, d(\mu-\nu).
\ee
Thanks to this fact, since $z\mapsto \nabla K(X_1(t,x,v)-z)$ is $B$-Lipschitz, we deduce that
$$
\left| \int_{\bt^d} \nabla K(X_1(t,x,v)-z)\di\bigl(\rho_{f_1(t)}(z)-\rho_{f_2(t)}(z)\bigr) \right| \leq B\,W_1(\rho_{f_1(t)},\rho_{f_2(t)}),
$$
and therefore
$$
T_1\leq B\,W_1(\rho_{f_1(t)},\rho_{f_2(t)})\int_{(\bt^d\times\br^d)^2}d\pi_0(x,v,y,w)=B\,W_1(\rho_{f_1(t)},\rho_{f_2(t)}).
$$
Let $\gamma_t=(X_1(t,\cdot,\cdot),X_2(t,\cdot,\cdot))_\#\pi_0 \in \Pi(\rho_{f_1},\rho_{f_2}).$ Then, by the definition of $W_1$ (see Definition~\ref{def:Wp}),
$$
W_1(\rho_{f_1(t)},\rho_{f_2(t)}) \le\int_{(\bt^d)^2} |x-y| d\gamma_t(x,y)=\int_{(\bt^d\times\br^d)^2} |X_1(t,x,v)-X_2(t,y,w)| d\pi_0(x,v,y,w),
$$
so using \eqref{eq:Q W1} we conclude that $T_1 \leq B(1+t)Q(t)$.

In conclusion, we proved that
$$
Q'(t)\leq 2B(1+t)^2Q(t), 
$$
therefore
$$
Q(t)\leq e^{\frac23 B\left((1+t)^3-1\right)}Q(0).
$$
Recalling \eqref{eq:Q W1 0} and \eqref{eq:Q W1}, this yields
\be
\label{eq:Mika}
W_1(f_1(t),f_2(t))\leq (1+t) e^{\frac23 B\left((1+t)^3-1\right)}W_1(f_1(0),f_2(0)).
\ee
As noted in the introduction, this estimate is more powerful than the usual Dobrushin's estimate \footnote{Dobrushin's argument is performed considering
the so-called bounded-Lipschitz distance on probability measures, which is defined by duality against bounded Lipschitz functions.
However, the same proof where one replaces the bounded-Lipschitz distance with the $W_1$ distance (which can be defined by duality against Lipschitz functions,
as shown in \eqref{eq:dual W1}),
 provides this bound.}
\be
\label{eq:Dobrushin}
W_1(f_1(t),f_2(t))\leq e^{(1+2B)t}W_1(f_1(0),f_2(0))
\ee
when $B$ is small.
On the other hand, for large times, the term $(1+t)^3$ in our estimate provides a worse bound \eqref{eq:Dobrushin}.
Hence, both bounds are helpful depending on the mutual sizes of $B$ and $t$, and one can choose to apply whichever gives the stronger bound.
In conclusion, one has
\be
\label{eq:Mika2}
W_1(f_1(t),f_2(t))\leq \min\left\{(1+t) e^{\frac23 B\left((1+t)^3-1\right)},e^{(1+2B)t}\right\}W_1(f_1(0),f_2(0)),
\ee
as desired.
\end{proof}

\section{Stability estimates for Vlasov-Poisson and quasi-neutral limits}
\subsection{The Vlasov-Poisson system}
The Vlasov-Poisson system is the classical kinetic model describing dilute, totally ionised, unmagnetized plasma. In its most common form, $f$ is the distribution function of the electrons moving in a self-induced electrostatic field, while the ions are assumed to act as a fixed background.
In this section, we consider the phase space to be $\T^d\times\R^d,$ for reasons that will be explained later.
\begin{equation}
\label{vp}
(VP):= \left\{ \begin{array}{ccc}\partial_t f+v\cdot \nabla_x f+ E\cdot \nabla_v f=0, \\
E=-\nabla U, \\
\Delta U=1- \int_{\R^d} f\, dv= 1- \rho_f,\\
f\vert_{t=0}=f_{0}\ge0,\ \ \int_{\T^d \times \R^d} f_{0}\,dx\,dv=1.
\end{array} \right.
\end{equation}

The well-posedness theory of this system has been extensively studied, see, for example, the survey paper \cite{GPI-proc-wp}. Global-in-time classical solutions have been constructed under various conditions on the initial data (see for example \cite{Bardos-Degond, BR, Lions-Perthame, Pfaffelmoser, Schaeffer, Ukai-Okabe}), while global-in-time weak solutions were presented in \cite{Arsenev} and \cite{Horst-Hunze} for $L^p$ initial data (see also \cite{BD85, BDG86}). 
In this section, we will focus on an important contribution to the uniqueness theory made by Loeper \cite{Loep}, who proved uniqueness for solutions of \eqref{vp} with bounded density by means of a strong-strong stability estimate in Wasserstein.

\subsection{Quasi-neutral limits}
Since plasmas are excellent conductors of electricity, and any charges that develop are readily neutralized, they can be treated as being quasi-neutral. On the other hand, at small spatial and time scales, the quasi-neutrality is no longer verified. The distance over which quasi-neutrality may break down can be described in terms of the Debye length $\lambda_D$, and varies according to the physical characteristics of the plasma.
% via the formula
%\begin{equation} \label{Def_Deb}
%\lambda_D : = \left (\frac{\epsilon_0 k_B T}{n q^2} \right )^{1/2},
%\end{equation}
%where $\epsilon_0$ is the vacuum permittivity, $k_B$ is the Boltzmann constant, $T$ is the temperature of the electrons, $n$ is their typical density and $q$ is the charge of one electron. 
The Debye length is usually much shorter than the typical observation scale. Therefore, we can define the parameter $\eps := \lambda_D/L$ and consider the limit as $\eps$ tends to zero. This procedure is known as {\em quasi-neutral limit}.

When we take the Debye length into account, in appropriate dimensionless variables, the Vlasov-Poisson system becomes:
\begin{equation} \label{Eqn_VP_eps}
(VP)_{\eps} : = \left\{ \begin{array}{ccc}\partial_t f_\eps+v\cdot \nabla_x f_\eps+ E_\eps\cdot \nabla_v f_\eps=0, \\
E_\eps=-\nabla_x U_\eps, \\
-\eps^2 \Delta_x U_\eps=\rho_{f_{\eps}} - 1 ,\\
f_\eps\vert_{t=0}=f_{0,\eps}\ge0,\ \ \int_{\mathbb{T}^d \times \mathbb{R}^d} f_{0,\eps}\,dx\,dv=1,
\end{array} \right. 
\end{equation}
and the energy of the rescaled system is the following:
\begin{equation}
\mathcal{E}(f_\eps(t)):= \frac{1}{2}\int_{\mathbb{T}^d \times \mathbb{R}^d} f_\eps |v|^2 dx dv + \frac{\eps^2}{2}\int_{\mathbb{T}^d}|\nabla_x U_\eps|^2 dx.
\end{equation}
The quasi-neutral limit corresponds to a singular limit for the rescaled system \eqref{Eqn_VP_eps}, in which the formal limiting system is the \emph{Kinetic Isothermal Euler} system:
\begin{equation} \label{Eqn_KIE}
(KIE) :=\left\{ \begin{array}{ccc}\partial_t f +v\cdot \nabla_x f+ E\cdot \nabla_v f=0, \\
E= -\nabla_x U, \\
\rho = 1,\\
f\vert_{t=0}=f_{0}\ge0,\ \ \int_{\mathbb{T}^d \times \mathbb{R}^d} f_0\,dx\,dv=1.
\end{array} \right.
\end{equation}
The force $E=-\nabla_x U$ is defined implicitly through the incompressibility constraint $\rho=1$, and may be thought of as a Lagrange multiplier associated to this constraint. In other words, electrons move under the effect of a gradient in such a way that their density remains equal to $1$ everywhere. Thus (KIE) is a ``kinetic'' version of the incompressible Euler equations.
As shown in \cite{Brenier89}, the potential $U$ formally satisfies the Laplace equation
$$
-\Delta_x U=  \sum_{i,j} \partial_{x_i}\partial_{x_j} \int_{\R^d} v_i v_j f \di\,v.
$$

As discussed in \cite{GPI-proc-qn}, the justification of this limit is very delicate. In particular, it can fail even for smooth initial data. 
Still, a series of positive results are available.
In particular, as shown in \cite{HKI2, HKI1}, a way to get the validity of the quasi-neutral limit for a large class of data can be achieved if one can prove some quantitative strong-strong stability at the level of the $(VP)_{\eps}$ system.
Also, the stronger the stability estimate, the larger the class of initial data for which the quasi-neutral limit hold.
In \cite{HKI2, HKI1} the authors prove that the quasi-neutral limit holds for initial data that are an extremely small perturbation of an analytic function.
Here, by introducing a suitable non-linear version of the Wasserstein distance, we can considerably improve that results.

Here is our main theorem, which provides us with a new $W_2$ stability estimate.
We prove the result with a general parameter $\e\leq 1$ as this is necessary for the study of the quasi-neutral limit.
The reader interested in the Vlasov-Poisson case can simply apply our estimate with $\e=1$.
\begin{theorem}
\label{thm:Loeper}
Let $\eps \leq 1$, and let $f_1, f_2$ be two weak solutions of the $(VP)_{\eps}$ system \eqref{Eqn_VP_eps}, and set 
$$
\rho_1:= \int_{\br^d} f_1 \, dv, \quad \rho_2= \int_{\br^d} f_2 \, dv.
$$ 
Define the function
\be
\label{eq:At}
A(t):=\|\rho_1(t)\|_{L^\infty(\mathbb T^d)}+\|\rho_2(t)\|_{L^\infty(\mathbb T^d)},
 \ee
 and assume that $A(t) \in L^1([0,T])$ for some $T>0$. There exist a dimensional constant $C_d>0$ and a universal constant $c_0>0$ such that the following holds:
if $W_2(f_1(0),f_2(0))$ is sufficiently small so that $W_2(f_1(0),f_2(0))\leq c_0 \eps$ and 
\begin{equation}\label{eq:main hyp}
\sqrt{\left|\log\left( \eps^{-2}W_2(f_1(0),f_2(0))^2 \left| \log \(\frac{1}2\eps^{-2}W_2(f_1(0),f_2(0))^2\)\right|\right)\right|} \geq \frac{C_d}{\eps}\int_0^TA(s)\,ds+\sqrt{\left|\log\left(\frac{\eps}e\right)\right|},
\end{equation}
then
$$
W_2(f_1(t),f_2(t))^2 \leq 2 e^{-\left(\sqrt{\left|\log\left\{ \eps^{-2}W_2(f_1(0),f_2(0))^2 \left| \log \(\frac{1}2\eps^{-2}W_2(f_1(0),f_2(0))^2\)\right|\right\}\right|} - \frac{C_d}{\eps}\int_0^tA(s)\,ds\right)^2} \qquad\text{for all }t \in [0,T].
$$
\end{theorem}
\begin{remark} The assumption \eqref{eq:main hyp} depends on the time interval $[0,T].$ If $T$ is very small so that 
$$
\frac{C_d}{\eps}\int_0^TA(s)\,ds\le\sqrt{\left|\log\left(\frac{\eps}e\right)\right|}
$$
then \eqref{eq:main hyp} 
corresponds to $W_2(f_1(0),f_2(0))\le \eps^3.$
Of course this is not the relevant regime since the time interval is usually at least of size $1$. 
In particular, since $A(t)\ge 2,$ \footnote{We recall that $\int_{\mathbb{T}^d}\rho_i(x,t)dx=1,$ that implies $\|\rho_i(\cdot, t) \|_{L^\infty(\mathbb{T}^d)}\ge 1,$ $i=1,2.$ Therefore $A(t)\ge 2.$ }
$$
\frac{C_d}{\eps}\int_0^TA(s)\,ds= \frac{C_T}{\eps} \qquad \text{for some constant $C_T\gtrsim 1$}.
$$
Therefore \eqref{eq:main hyp} corresponds to asking $W_2(f_1(0),f_2(0))$ being bounded by $e^{-C\eps^{-2}}.$ This requirement is very natural in this context, as also discussed in Remark~\ref{rem:double}.
\end{remark}
As in \cite{HKI2, HKI1}, Theorem~\ref{thm:Loeper} yields the validity of the quasi-neutral limit for $W_2$-perturbations of analytic data. However, our estimate is stronger with respect to the previous results and provides an almost optimal rate in the quasi-neutral limit. More broadly, we believe that our approach for proving Theorem~\ref{thm:Loeper} has its own interest and could be used in other settings.

To state our application to the quasi-neutral limit, we need to recall some notation introduced by Grenier \cite{Grenier96} in one of the first mathematical works on this topic. In \cite{Grenier96} the author relies on an interpretation of the plasma as a superposition of a -possibly uncountable- collection of fluids and he shows that the quasi-neutral limit holds when the sequence of initial data $f_{0,\e}$ enjoys uniform analytic regularity with respect to the space variable.  As explained in \cite{HKI2} (see the discussion after Definition $1.4$), this decomposition is purely a technical tool and it does not impose any restriction on the initial datum.
This result has been improved by Brenier \cite{Br00}, who gives a rigorous justification of  the quasi-neutral limit in the so called ``cold electron'' case, i.e. when the initial distribution $f_{0,\e}$ converges to a monokinetic profile 
$$
f_0(x,v) = \rho_0(x) \delta_{v= v_0(x)}
$$ 
where $\delta_v$ denotes the Dirac measure in velocity, see also \cite{Br00,Mas,GSR}. 

Let us define a suitable analytic norm, as in \cite{Grenier96}: given $\delta>0$ and a function $g:\bt^d \to \mathbb R$, we define
$$
\| g \|_{B_\delta} := \sum_{k \in \bz^d} | \widehat{g}(k) | \delta^{|k|},
$$
where $\widehat{g}(k)$ is the $k$-th Fourier coefficient of $g$. We define $B_\delta$ as the space of functions $g$ such that $\| g \|_{B_\delta}<+\infty$.

\begin{corol}
\label{cor1}
Let $d=2, 3$, and let $\gamma$, $\delta_0$, %$\eta>0$ 
and $C_0$ be positive constants.
Consider a sequence $(f_{0,\e})$ of non-negative initial data in $L^1$ for \eqref{Eqn_VP_eps} such that for all $\e \in (0,1)$, and all $x \in \bt^d$, 
\begin{itemize}
\item (uniform estimates) 
$$
\| f_{0,\e}\|_\infty\leq C_0, \quad \mathcal{E}(f_{0,\e}) \leq C_0,
$$
\item (compact support in velocity)
$$
f_{0,\e}(x,v)= 0 \quad \text{if } |v| > \frac{1}{\e^\gamma},
$$
\item (analytic + perturbation)
Assume the following decomposition:
$$
f_{0,\e} = g_{0,\e} + h_{0,\e},
$$
where $(g_{0,\e})$ is a sequence of continuous functions
satisfying
$$
\sup_{\e\in (0,1)}\sup_{v \in \br^d} \, (1+|v|^2) \| g_{0,\e} (\cdot,v)\|_{B_{\delta_0}} \leq C_0,
$$
admitting a limit $g_0$ in the sense of distributions.
Furthemore, $(h_{0,\e})$ is a sequence of functions %in $L^1 \cap L^\infty$ with finite first moment, 
satisfying for all $\e>0$
\be
\label{eq:K}
W_2(f_{0,\e},g_{0,\e}) \leq e^{-K \eps^{-2\zeta}}\qquad \text{with }\zeta=\left\{
\begin{array}{ll}
1+2\max\{\beta,\gamma\}) & \text{if }d=2,\\
1+\max\{38,3\gamma\}) & \text{if }d=3,\\
\end{array}
\right.
\ee
for some constants $K>0$ and $\beta>2$.
\end{itemize}
For all $\e \in (0,1)$, consider $f_\e(t)$ a global weak solution of \eqref{Eqn_VP_eps} with initial condition $f_{0,\e}$, and define the filtered distribution function
\begin{equation} 
\widetilde{f}_\e(t,x,v) := f_\e \Big(t,x,v-\frac{1}{i}(d_+(t,x)e^{\frac{it}{\sqrt \e}}-d_-(t,x)e^{-\frac{it}{\sqrt \e}})\Big)
\end{equation}
where $(d_\pm)$ are the correctors are defined as the solution of
\begin{equation}
\label{eq:correctors}
{\rm curl} \ d_{\pm}=0, \qquad {\rm div} \bigg(\pt_t d_{\pm}+\left(\int \rho_\theta v_\theta \mu (d\theta) \cdot \na \right)d_{\pm} \bigg) =0,
\end{equation}
\begin{equation}
\label{eq:correctors_initial}
{\rm div} d_{\pm} (0)= \underset{\e \to 0}{\lim} {\rm div} \frac{\sqrt \e E^\e(0) \pm i j^\e(0)}{2},\qquad j^\e:= \int \rho^\e_\theta v^\e_ \theta \mu (d\theta).
\end{equation}
Then there exist $T>0$, and $g(t)$ a weak solution on $[0,T]$ of \eqref{Eqn_KIE} with initial condition $g_0,$ such that
$$
\lim_{\e \to 0} \sup_{t \in [0,T]} W_1(\widetilde{f}_\e(t), g(t)) = 0.
$$

\end{corol}

\begin{remark}\label{rem:double} Already in the one dimensional case, there is a negative result stating that an initial rate of convergence of the form $W_2(f_{0,\e},g_{0,\e}) \leq \e^k$ for some $k>0$ is not sufficient to ensure the validity of the quasi-neutral limit for positive times. This is the consequence of \emph{instability mechanisms} described in \cite{Grenier99} and \cite{HKH}. Hence, our assumption on the size of $W_2(f_{0,\e},g_{0,\e}) $ considerably improves the results in \cite{HKI2, HKI1}, where a double exponential $\exp\big(-\exp({K \eps^{-\zeta}})\big)$ was required.
\end{remark}

\begin{remark} In Corollary~\ref{cor1} we consider sequences of initial conditions with compact support in velocity (yet, we allow the support to grow polynomially as $\e$ goes to zero).
The reason is that
we need $L^\infty$ bounds on the density $\rho_{f_\e}(t)=\int f_\e(t)\,dv$, so a control on the support in velocity is needed. 
We have decided to put these assumptions because they are the same as in \cite{BR} and so we can rely on some estimates proved in that paper.
However, using the argument in \cite{Pallard} (see also \cite{GPI-WP}) one could relax the assumptions and require only a moment condition on $f_{0,\eps}$. Providing this extension is not difficult, but it would require some work that would go beyond
the main goal of this paper.
\end{remark}

Before proving Proof of Theorem~\ref{thm:Loeper}, we first show how it implies Corollary~\ref{cor1}.

\begin{proof}[Proof of Corollary~\ref{cor1}]
Let $g_\e(t)$ denote the solution of \eqref{Eqn_VP_eps} starting from $g_{0,\e}$. As shown in \cite[Section 4]{HKI2}, under the assumptions in the statement, the following bounds hold:
$$
\|\rho_{g_\e}\|_{L^\infty([0,T]\times \bt ^d)}\leq \bar C,\qquad \|\rho_{f_\e}\|_{L^\infty([0,T]\times \bt ^d)}\leq \bar C\e^{-(\zeta-1)},
$$
where $\zeta=\zeta(d)$ is as in the statement.
Hence Theorem~\ref{thm:Loeper} and \eqref{eq:K} yield
$$
\sup_{[0,T]}W_2(f_\e(t),g_\e(t)) \leq 2 e^{-\left(\sqrt{\left|\log\left\{ \eps^{-2}W_2(f_{0,\e},g_{0,\e})^2 \left| \log \(\frac{1}2\eps^{-2}W_2(f_{0,\e},g_{0,\e})^2\)\right|\right\}\right|} - C_d\bar C T\e^{-\zeta} \right)^2}
$$
provided $C_d^2\bar C^2T^2<K$ (which can be guaranteed by taking $T$ smaller if necessary).
This implies that $\sup_{[0,T]}W_2(f_\e(t),g_\e(t)) \to 0$ as $\e \to 0,$
and we can now conclude as in \cite[Proof of Theorem 1.7]{HKI2}.
\end{proof}
\subsection{Proof of Theorem~\ref{thm:Loeper}}

Before starting the proof we recall \cite[Lemma 3.2]{HKI2}, see also \cite[Lemma 3.3]{GPI-WP}.
%(notice that $|x-y| \leq \sqrt{d}$ for all $x,y \in \mathbb T^d$): 
\begin{lemma}
\label{lem:Loeper}
Let $\Psi_i:\mathbb T^d\to \mathbb R$ solve
$$
-\e^2 \Delta \Psi_i=\rho_i-1,\qquad i=1,2.
$$
Then
$$
 \e^2 \|\nabla \Psi_1-\nabla \Psi_2\|_{L^2(\mathbb T^d)} \leq \Bigl[\max\bigl\{\|\rho_1\|_{L^\infty(\mathbb T^d)},
 \|\rho_2\|_{L^\infty(\mathbb T^d)}\bigr\} \Bigr]^{1/2}\,W_2(\rho_1,\rho_2),
$$
$$
\e^2|\nabla \Psi_i(x) - \nabla \Psi_i(y)| \leq C\,|x-y|\,\log\biggl(\frac{4 \sqrt{d}}{|x-y|} \biggr) \,\|\rho_i-1\|_{L^\infty(\mathbb T^d)}\qquad\forall\,x,y \in\mathbb T^d,\, i=1,2.
$$
\end{lemma}

Let $(X_i,V_i)$ denote the characteristic flow associated to $f_i$, that is
$$
\left\{
\begin{array}{l}
\dot X_i(t,x,v)=V_i(t,x,v),\\
\dot V_i(t,x,v)=E_i(t,X_i(t,x,v)),\\
X_i(0,x,v)=x,\,\,V_i(0,x,v)=v,
\end{array}
\right.
\qquad E_i=\nabla U_i,\qquad \eps^2\Delta U_i=\rho_{f_i}-1.
$$
To prove Theorem~\ref{thm:Loeper}, we consider $\pi_0$ an optimal $W_2$-coupling between $f_{1}(0)$ and $f_{2}(0)$,
and we define the quantity $Q(t)$ defined as the unique constant (assuming it exists) such that 
$$
Q(t)= \frac{1}{2}\int_{(\bt^d\times\br^d)^2} \[\eps^{-2}|\log Q(t)| \,|X_1(t,x,v)-X_2(t,y,w)|^2+|V_1(t,x,v)-V_2(t,y,w)|^2\]d\pi_0(x,v,y,w).
$$
In other words, we are considering a quantity of the form 
$$
Q(t)= \frac{1}{2}\int_{(\bt^d\times\br^d)^2} \[\lambda(t)|X_1(t,x,v)-X_2(t,y,w)|^2+|V_1(t,x,v)-V_2(t,y,w)|^2\]d\pi_0(x,v,y,w),
$$
with $\lambda(t)$ depending on time, and we are assuming that actually $\lambda(t)$ is a function of $Q(t)$ itself.
The particular choice $\lambda(t)=\eps^{-2}|\log Q(t)|$ is specific to this problem: the logarithm will help to compensate for the log-Lipschitz regularity
of the electric fields, while $\eps^{-2}$ is the natural scaling in the current setting.

Note that a priori is not clear that $Q(t)$ is well-defined. This will be proved in Lemma~\ref{lem:well def} below.
However, assuming for now that $Q(t)$ is well-defined, we show how this quantity allows us to prove the result.
We have 
\begin{align}
Q'(t)&=\frac{1}{2}\int_{(\bt^d\times\br^d)^2}\lambda'(t)|X_1(t,x,v)-X_2(t,y,w)|^2\,d\pi_0(x,v,y,w)\\
&\qquad +\int_{(\bt^d\times\br^d)^2} \[\lambda(t)(X_1(t,x,v)-X_2(t,y,w))\cdot(V_1(t,x,v)-V_2(t,y,w)\]d\pi_0(x,v,y,w)\\
&\qquad -\int_{(\bt^d\times\br^d)^2}\[(V_1(t,x,v)-V_2(t,y,w)\cdot(E_1(t, X_1(t,x,v))-E_2(t, X_2(t,y,w)))\]d\pi_0(x,v,y,w)\\
\end{align}
By Cauchy-Schwartz inequality and recalling the definition of $Q(t)$ we have:
\begin{align}
Q'(t)&\le\frac{1}{2}\lambda'(t)\int_{(\bt^d\times\br^d)^2}|X_1(t,x,v)-X_2(t,y,w)|^2\,d\pi_0(x,v,y,w)\\
&\qquad +\lambda(t)\(\int_{(\bt^d\times\br^d)^2}|X_1(t,x,v)-X_2(t,y,w)|^2\,d\pi_0(x,v,y,w)\)^{\frac{1}{2}}\cdot\\
&\qquad\qquad\qquad \cdot \(\int_{(\bt^d\times\br^d)^2}|V_1(t,x,v)-V_2(t,y,w)|^2\,d\pi_0(x,v,y,w)\)^{\frac{1}{2}}\\
&\qquad +\(\int_{(\bt^d\times\br^d)^2}|V_1(t,x,v)-V_2(t,y,w)|^2\,d\pi_0(x,v,y,w)\)^{\frac{1}{2}}\cdot\\
&\qquad\qquad\qquad\cdot \(\int_{(\bt^d\times\br^d)^2}|E_1(t, X_1(t,x,v))-E_2(t, X_2(t,y,w))|^2\,d\pi_0(x,v,y,w)\)^{\frac{1}{2}}\\
&\le \frac{1}{2}\lambda'(t)\int_{(\bt^d\times\br^d)^2}|X_1(t,x,v)-X_2(t,y,w)|^2\,d\pi_0(x,v,y,w)\\
&+2\sqrt{\lambda(t)}Q(t)+\sqrt{ Q(t)}\| E_1(t, X_1(t,x,v))-E_2(t, X_2(t,y,w))\|_{L^2(d\pi_0(x,v,y,w))}.
\end{align}
Adding and subtracting $-E_2(t,X_1)$ we obtain:
\begin{align}\label{eq:Q'}
Q'(t)&\le \frac{1}{2}\lambda'(t)\int_{(\bt^d\times\br^d)^2}|X_1(t,x,v)-X_2(t,y,w)|^2\,d\pi_0(x,v,y,w)\\
&+2\sqrt{\lambda(t)}Q(t)
+\sqrt{Q(t)}\(T_1+T_2\)
\end{align}
where
$$
T_1=\| E_2(t, X_1(t,x,v))-E_2(t, X_2(t,y,w))\|_{L^2(d\pi_0(x,v,y,w))},$$
$$
 T_2=\| E_1(t, X_1(t,x,v))-E_2(t,X_1(t,x,v))\|_{L^2(d\pi_0(x,v,y,w))}.
$$
Thanks to Lemma~\ref{lem:Loeper} and by the very same argument in \cite{HKI2} we can bound $T_1$ and $T_2$ as follows: \footnote{Note that, since $\rho_i\ge0$ and $\|\rho_i(\cdot, t) \|_{L^\infty(\mathbb{T}^d)}\ge 1,$ then $\|\rho_i(\cdot, t)-1 \|_{L^\infty(\mathbb{T}^d)}\le \|\rho_i(\cdot, t) \|_{L^\infty(\mathbb{T}^d)}\le A(t)$ for $i=1,2.$}
$$
T_2\le \frac{C}{\e^2} A(t)\sqrt{\frac{Q(t)}{\lambda(t)}},
\qquad\text{and}\qquad
T_1\le \frac{C}{\e^2}A(t) \sqrt{\phi\(\frac{Q(t)}{\lambda(t)}\)}
$$
where we have 
\begin{equation}
\phi(s)=\left\{
\begin{array}{ll}
 s\log^2(s) & \mbox{for}\, s\in(0,1/e]\\
s& \mbox{for}\, s>1/e
\end{array}
\right.
\end{equation}
We now recall that $\lambda(t)=\e^{-2}|\log(Q(t))|$ and we substitute this expression in the derivative of $Q(t).$
Notice that in this estimate we are interested in small values of $Q(t)$ and in particular, as we will show below, we will always be in the regime $\eps^2 Q(t)/|\log(Q(t))|\in (0,1/e).$
%Moreover $|\log(Q(t))|=-\log(Q(t)).$
Therefore we have
$$
T_1\le \frac{C}{\e^2}A(t)\sqrt{\frac{\e^2\,Q(t)}{|\log(Q(t))|}\log^2 \(\frac{\e^2\,Q(t)}{|\log(Q(t))|}\)},
$$
%\begin{align}
%Q'(t)&\le C \(\frac{\lambda'(t)}{\lambda(t)}+\sqrt{\lambda(t)}\)Q(t)+{\frac{Q(t)}{\e^2\sqrt{\lambda(t)}}}+\frac{\sqrt{Q(t)}}{\e^2}\sqrt{\frac{Q(t)}{\lambda(t)}\log^2 \(\frac{Q(t)}{\lambda(t)}\)}\\
%&\le C \(\frac{\lambda'(t)}{\lambda(t)}+\sqrt{\lambda(t)}+\frac{1}{\e^2\sqrt{\lambda(t)}}\)Q(t)+\frac{\sqrt{Q(t)}}{\e^2}\sqrt{\frac{Q(t)}{\lambda(t)}\log^2 \(\frac{Q(t)}{\lambda(t)}\)}.\\
%\end{align}
%
so by equation~\eqref{eq:Q'} we have
\begin{multline*}
Q'(t)\le \(-\frac{1}{2}\frac{Q'(t)}{|\log (Q(t))|}\int_{(\bt^d\times\br^d)^2}|X_1(t,x,v)-X_2(t,y,w)|^2\,d\pi_0(x,v,y,w)\)\\
+\(2\frac{\sqrt{|\log(Q(t))|}}{\eps}+\frac{C\,A(t)}{\e \sqrt{|\log(Q(t))|}}\)Q(t)+C\,A(t)\frac{\sqrt{Q(t)}}{\e}\sqrt{\frac{Q(t)}{|\log(Q(t))|}\log^2 \(\frac{\eps^2 \,Q(t)}{|\log(Q(t))|}\)}.
\end{multline*}
We now consider two cases, depending on the sign of $Q'(t)$. If $Q'(t)\leq 0$, then we do not do anything. If instead $Q'(t)>0$, then the first term in the right-hand side above is negative, and therefore
\begin{align}
Q'(t)\le \(2\frac{\sqrt{|\log(Q(t))|}}{\eps}+\frac{C\,A(t)}{\e \sqrt{|\log(Q(t))|}}\)Q(t)+C\,A(t)\frac{\sqrt{Q(t)}}{\e}\sqrt{\frac{Q(t)}{|\log(Q(t))|}\log^2 \(\frac{\eps^2 \,Q(t)}{|\log(Q(t))|}\)}.
\end{align}
Since the right-hand side above is nonnegative, independently of the sign of $Q'(t)$ we know that the bound above holds.
We now observe that as long as $Q(t)\leq \eps$ then 
$$
\log^2 \(\frac{\eps^2 \,Q(t)}{|\log(Q(t))|}\) \leq C\log^2(Q(t)).
$$
Thus, 
$$
Q'(t)\le \(2\frac{\sqrt{|\log(Q(t))|}}{\eps}+\frac{C\,A(t)}{\e \sqrt{|\log(Q(t))|}}+\frac{C\,A(t)\sqrt{|\log(Q(t))|}}{\e }\)Q(t) \qquad \text{provided $Q(t)\leq \eps$}.
$$
Since $A(t)\geq 1$ \footnote{We recall that $\int_{\mathbb{T}^d}\rho(x,t)dx=1,$ that implies $\|\rho(\cdot, t) \|_{L^\infty(\mathbb{T}^d)}\ge 1.$ Therefore $A(t)\ge1.$}, provided $|\log(Q(t))|\geq 1$, the above bound reduces to
$$
Q'(t)\le \frac{2C_d\,A(t)}{\e }Q(t)\sqrt{|\log(Q(t))|} 
$$
where $C_d$
 is a dimensional constant.
 Note that the two conditions $Q(t)\leq \eps$ and $|\log(Q(t))|\geq 1$ are guaranteed if $Q(t)\leq \frac{\eps}{e}$ (recall that $\e\leq 1$ by assumption).
 
Hence, provided that we are in the regime $Q(s)\leq \frac{\eps}{e}$ on $[0,t]$, this implies
\be
\label{eq:main}
Q(t) \leq R(t):=e^{-\left(\sqrt{|\log(Q(0))|} - \frac{C_d}{\eps}\int_0^tA(s)\,ds\right)^2}.
\ee
We observe that the bound \eqref{eq:main} guarantees that $$\sup_{s\in[0,t]}Q(s)\leq \frac{\eps}{e}\quad\text{ holds if } \quad \sup_{s \in [0,t]}R(s)\leq \frac{\eps}e.$$
In particular, \eqref{eq:main} holds if 
\be
\label{eq:provided}
\sqrt{|\log(Q(0))|}\geq \frac{C_d}{\eps} \int_0^tA(s)\,ds+\sqrt{\left|\log\left(\frac{\eps}e\right)\right|}
\ee
We now compare the quantity $Q$ to the Wasserstein distance.
First of all, since $Q(t)\leq \frac{\eps}{e}$ then $\eps^{-2}|\log(Q(t))| \geq 1$, therefore
\begin{align}
\frac12 W_2(f_1(t),f_2(t))^2&\leq \frac{1}{2}\int_{(\bt^d\times\br^d)^2} \[ |X_1(t,x,v)-X_2(t,y,w)|^2+|V_1(t,x,v)-V_2(t,y,w)|^2\]d\pi_0(x,v,y,w) \\
&\leq Q(t).
\end{align}
On the other hand, since $\eps^{-2}|\log(Q(0))| \geq 1$ and $\pi_0$ is an optimal plan,
$$
Q(0)\leq \frac{1}{2}\eps^{-2}|\log(Q(0))| \int_{(\bt^d\times\br^d)^2} \[|x-y|^2+|v-w|^2\]d\pi_0(x,v,y,w)=\frac{1}{2}\eps^{-2}|\log(Q(0))|W_2(f_1(0),f_2(0))^2,
$$
or equivalently
$$
\frac{Q(0)}{|\log(Q(0))|}\leq \frac{1}2\eps^{-2}W_2(f_1(0),f_2(0))^2,
$$
We now observe that, near the origin, the inverse of the function $s\mapsto \frac{s}{|\log s|}$ behaves like $\tau\mapsto \tau |\log \tau|$. In particular, there exists a universal small constant $c_0>0$ such that
$$
\frac{s}{|\log s|}\leq \tau \qquad \text{for some }0 \leq \tau \leq c_0 \qquad \Rightarrow \qquad s \leq 2 \tau|\log \tau|.
$$
Hence, if $ \frac{1}2\eps^{-2}W_2(f_1(0),f_2(0))^2 \leq c_0$, we deduce that 
$$
Q(0)\leq \eps^{-2}W_2(f_1(0),f_2(0))^2 \left| \log \(\frac{1}2\eps^{-2}W_2(f_1(0),f_2(0))^2\)\right|.
$$
Combining these bounds with \eqref{eq:main}, and recalling \eqref{eq:provided}, this implies
$$
W_2(f_1(t),f_2(t))^2 \leq 2 e^{-\left(\sqrt{\left|\log\left( \eps^{-2}W_2(f_1(0),f_2(0))^2 \left| \log \(\frac{1}2\eps^{-2}W_2(f_1(0),f_2(0))^2\)\right|\right)\right|} - \frac{C_d}{\eps}\int_0^tA(s)\,ds\right)^2} 
$$
provided $ \frac{1}2\eps^{-2}W_2(f_1(0),f_2(0))^2 \leq c_0$ and
$$
\sqrt{\left|\log\left( \eps^{-2}W_2(f_1(0),f_2(0))^2 \left| \log \(\frac{1}2\eps^{-2}W_2(f_1(0),f_2(0))^2\)\right|\right)\right|} \geq \frac{C_d}{\eps}\int_0^tA(s)\,ds+\sqrt{\left|\log\left(\frac{\eps}e\right)\right|}.
$$
%In particular, since by assumption $\eps^{-1}W_2(f_1(0),f_2(0))$ is small, then we can bound
%$$
%\sqrt{\left|\log\left( \eps^{-2}W_2(f_1(0),f_2(0))^2 \left| \log \(\frac{1}2\eps^{-2}W_2(f_1(0),f_2(0))^2\)\right|\right)\right|} \geq \frac54 \sqrt{|\log\left(W_2(f_1(0),f_2(0)) \right)|}.
%$$
%Thus, provided $\frac{C_d}{\eps}\int_0^T A(s)\,ds \leq \frac14 \sqrt{|\log\left(W_2(f_1(0),f_2(0)) \right)|}$, we get
%$$
%W_2(f_1(t),f_2(t))^2 \leq 2 e^{- \bigl(\sqrt{|\log\left(W_2(f_1(0),f_2(0)) \right)|} \bigr)^2}=2 W_2(f_1(0),f_2(0)) \qquad \text{for all }t \in [0,T].
%$$

Finally, to complete the proof, we show the following:
\begin{lemma} 
\label{lem:well def}
With the notation and assumptions of the theorem, the quantity $Q(t)$ is well defined and it is locally Lipschitz continuous where $Q(t)>0$.
In particular it is differentiable a.e.
\end{lemma}
\begin{proof}
Set
$$
D(t):=\frac{1}{2}\int_{(\bt^d\times\br^d)^2} |X_1(t,x,v)-X_2(t,y,w)|^2 d\pi_0(x,v,y,w), $$
$$
 E(t):=\frac{1}{2}\int_{(\bt^d\times\br^d)^2}|V_1(t,x,v)-V_2(t,y,w)|^2 d\pi_0(x,v,y,w).
$$
We can assume that $D(t)$ and $E(t)$ are nonzero, otherwise we are in the ``degenerate'' situation where $f_1\equiv f_2$, in which case $Q(t)$ is trivially 0. Also, since $D(t)$ and $E(t)$ are written in terms of the characteristic flow, it is standard to check that they are differentiable. \footnote{There is no novelty here, as these are the quantities that appear also in \cite{Loep}, where Loeper computes their derivatives and show that they can be controlled in terms of $D(t)$ and $E(t)$ themselves. In particular, the quantities $D(t)$ and $E(t)$ are also uniformly Lipschitz.}

We note that the quantity $Q(t)$ is implicitly defined via the relation
\begin{equation}
\label{eq:QDE}
Q(t)= \eps^{-2}|\log Q(t)| D(t)+E(t),
\end{equation}
or equivalently, for each fixed time $t$, $Q(t)$ is the solution of the equation
$$
F(q, D(t), E(t))=0 \qquad \text{with } F(q, r , s):=q+\eps^{-2} \log q\,r-s\qquad \text{for } q \in (0,1).
$$
Since the function $q\mapsto q+\eps^{-2}\log q\, D(t)$ is strictly increasing on $(0,1) $ and its image covers the interval $(0,1)$,
we deduce that the equation above has a unique solution provided $E(t)< 1$. 
Hence, this proves that $Q(t) \in (0,1)$ is well defined provided $E(t)< 1$. 
In addition, thanks to the implicit function theorem applied to the function $F\in C^1_{\rm loc}((0,1)\times \R \times \R)$, we deduce the existence of a $C^1_{\rm loc}$ function $G$ such that $Q(t)=G(D(t),E(t))$.
Now, differentiating the relation \eqref{eq:QDE} with respect to $t$ we obtain
$$
Q'(t)\biggl(1+\eps^{-2}\frac{D(t)}{Q(t)}\biggr)=\eps^{-2}|\log Q(t)| D'(t)+E'(t).
$$
Hence, since $D$ and $E$ are uniformly bounded and Lipschitz, for any $\delta >0$ we deduce that
$$
|Q'(t)| \leq \frac{\eps^{-2}|\log Q(t)|\,| D'(t)|+|E'(t)|}{1+\eps^{-2}\frac{D(t)}{Q(t)}} \leq C_\delta \qquad \text{where }Q(t)>\delta.
$$
This proves that, for any $\delta>0$, the function  $t\mapsto Q(t)$ is uniformly Lipschitz continuous inside the set $\{Q(t)>\delta\}$. This proves that  $Q(t)$ is locally Lipschitz continuous inside the region $\{Q(t)>0\}$.

So, to conclude the proof, we need to ensure that  $E(t)< 1$.
Note that, since by assumption $E(0)\leq\frac{1}{2} W_2(f_1(0),f_2(0))^2 \ll 1$, by continuity we have that $E(t)<1$ for $t>0$ small. So $Q(t)$ is well defined for $t>0$ small.
Also, as long as $Q(t)$ is well defined, we have that $Q(t)\geq E(t)$. Hence, as long as $Q(t)$ is well defined, we have that 
$$
E(t)\leq Q(t) \leq e^{-\left(\sqrt{|\log(Q(0))|} - \frac{C_d}{\eps}\int_0^tA(s)\,ds\right)^2}.
$$
Since, by our smallness assumption on $W_2(f_1(0),f_2(0))$, the right hand side above remain small on $[0,T]$, the bound above guarantees that $E(t)\ll 1$ for all $t \in [0,T]$.
This proves that $Q(t)$ is well-defined on $[0,T]$, which concludes the proof. 
\end{proof}
\begin{remark}
\label{rmk:Phi}
In the previous proof we considered $\lambda(t)=\eps^{-2}|\log(Q(t))|$ and in Lemma~\ref{lem:well def} we proved that $Q(t)$ is well-defined provided it is small enough. This restriction is due to the fact that the function $\R^+ \ni s\mapsto \eps^{-2}|\log s|$ is decreasing only for $s\in(0,1).$ An alternative choice could have been to define 
$$
\Phi(s)=\left\{
\begin{array}{ll}
|\log s|& \mbox{for}\, s\in(0,1/e]\\
e^{-1}s^{-1}& \mbox{for}\, s>1/e,
\end{array}
\right.
$$
and $\lambda(t):= \eps^{-2}\Phi(Q(t)).$ With this choice, since $\R^+ \ni s\mapsto \Phi(s)$ is decreasing and of class $C^1$, one can define $Q(t)$ as the unique solution of 
$$
F(q, D(t), E(t))=0 \qquad \text{with } F(q, r , s):=q+ \eps^{-2}\Phi(q)\,r-s\qquad \text{for } q \in (0,\infty).
$$
With this definition, the proof of Lemma~\ref{lem:well def} shows that $Q(t)$ is always well defined (without any restriction on the size of $E(t)$), and it is locally Lipschitz continuous where $Q(t)>0$.

Since in our setting we are interested in the case $Q(t)\ll 1,$ there is no advantage in using this latter definition of $\lambda.$ However this observation could be useful in other situations, see also Section~\ref{sect:summary} below.
\end{remark}

\bigskip

\section{Summary, generalizations, and perspectives}
\label{sect:summary}

As we have seen in the last two sections, suitably modifying Wasserstein distances can be particularly useful in a kinetic setting to take advantage of the asymmetry between $x$ and $v$.
More precisely, let $\mathcal{X}=\mathbb T^d$ or $\mathcal{X}=\mathbb R^d$, and let $\mu$ and $\nu$ be two probability measures on $\mathcal{X}\times\mathbb{R}^d$.
Also, let $\Pi(\mu,\nu)$ denote the collection of all measures on $(\mathcal{X} \times \mathbb{R}^d)^2$ with marginals $\mu$ and $\nu$ on the first and second factors respectively.

The first natural generalization, given $p\geq 1$ and $\lambda\in \mathbb{R}^+$, is to consider
$$
\W_{\lambda,\, p} (\mu, \nu):=\left( \inf_{\pi \in \Pi(\mu,\nu)} \int_{(\mathcal{X} \times \mathbb{R}^d)^2} \(\lambda|x-y|^{p} +|v-w|^p\)\mathrm{d} \pi (x,v,y,w) \right)^{1/p},
$$
as done in \cite{GPI-SIAM2018,GPI20, Laz}.

An alternative way, introduced in \cite{Salem2021} for $p=2$, would be to consider three parameters $a,b,c>0$ such that $\sqrt{ac}>b$ and define
$$
\W_{a,b,c,\, p} (\mu, \nu):=\left( \inf_{\pi \in \Pi(\mu,\nu)} \int_{(\mathcal{X} \times \mathbb{R}^d)^2} \(a |x-y|^{2} +2b(x-y)\cdot (v-w)+c|v-w|^2\)^{p/2}\mathrm{d} \pi (x,v,y,w) \right)^{1/p},
$$

In this paper, we have introduced two different generalizations.
\begin{itemize}
\item[(i)] First, we considered the nonlinear version of the $\W_{\lambda,\, p} (\mu, \nu)$ by choosing $\lambda$ depending on the distance itself.
We defined this along a flow, but that can be also be defined in a general setting as follows:\\
given $p\geq 1$ and a decreasing function $\Phi:\R^+\to \R^+$, for every $\pi \in \Pi(\mu,\nu)$ and $\lambda$ we define $D_p(\pi,\Phi)$ as the unique number $s$ such that
$$
s- \Phi(s)\int_{(\mathcal{X} \times \mathbb{R}^d)^2} |x-y|^{p}\mathrm{d} \pi (x,v,y,w) = \int_{(\mathcal{X} \times \mathbb{R}^d)^2} |v-w|^p\mathrm{d} \pi (x,v,y,w)
$$
(arguing as in the proof of Lemma~\ref{lem:well def} it is easy to check that $D(\pi,\Phi)$ is well defined, see also Remark~\ref{rmk:Phi}).
Then, we set
$$
\W_{\Phi,p}(\mu,\nu):=\left( \inf_{\pi \in \Pi(\mu,\nu)} D_p(\pi,\Phi) \right)^{1/p}.
$$
This definition with $\Phi(s)=\eps^{-2}|\log s|$ for $s \in (0,1/e)$ and $p=2$ essentially corresponds to the quantity used in the proof of Theorem~\ref{thm:Loeper}, although there we considered the quantity $D(t)$ where we did not take the infimum over couplings $\pi \in \Pi(\mu,\nu),$ since it was not needed for our purpose.
\item[(ii)]
In a different direction, we modified the $W_1$ distance by introducing a shift in position. Note that this second quantity cannot be defined as a ``static'' distance since the shift $x -tv$ depends on the time $t$. Hence, one can generalize it only as a time dependent quantity as follows:
$$
\widetilde \W_{t,\, p} (\mu, \nu):=\left( \inf_{\pi \in \Pi(\mu,\nu)} \int_{(\mathcal{X} \times \mathbb{R}^d)^2} \big(|(x-tv)-(y-tw)|^{p} +|v-w|^p\big)\mathrm{d} \pi (x,v,y,w) \right)^{1/p}.
$$
\end{itemize}
Of course, these approaches can be further combined by mixing the different quantities defined above.
Note that there is no universal ``best'' choice, and each problem requires its adaptation. Still, we believe, as this paper shows, that this approach can lead to an improvement to several existing results, as well as to prove new estimates.
In addition, the approach is very general and can be useful in any situation where there is an asymmetry between the variables involved.

To mention some concrete applications, our ideas could also be applied in the setting of quantum systems by suitably modifying the quantum Wasserstein distances introduced in \cite{GolseMouhotPaul2016, GolsePaul2019}.
Also, our new Loeper-type estimate may be helpful to obtain stability estimates in $W_2$ when the density belongs to some suitable Orlicz spaces, in analogy to \cite{HoldingMiot2018} where stability estimates have been proved for $W_1$.

\section*{Acknowledgments}
We are grateful to Megan Griffin-Pickering and Evelyne Miot for their valuable comments on a preliminary version of this paper. We also thank the anonymous referees for their useful comments and observations.

\end{document}